\documentclass[11pt]{article} 

\usepackage[utf8]{inputenc} 

\usepackage{geometry} 
\usepackage{graphicx} 
\usepackage{graphics}

\usepackage{booktabs} 
\usepackage{array} 
\usepackage{paralist} 
\usepackage{verbatim} 
\usepackage{subfig} 
\usepackage{amssymb}
\usepackage{amsthm}
\usepackage{amsmath,amssymb,amsopn,amsfonts,mathrsfs,amsbsy,amscd}
\usepackage{longtable}
\usepackage[dvipsnames]{xcolor}
\usepackage{ulem}
\usepackage{float}
\usepackage{tabularx}
\usepackage{caption}
\usepackage{bbm}
\usepackage{accents}
\usepackage{stmaryrd}

\usepackage{mathtools} 
\usepackage{etoolbox}
\usepackage{multirow}
\usepackage{enumerate}
\usepackage{url}
\usepackage[export]{adjustbox}
\usepackage{palatino}

\usepackage{imakeidx}
\makeindex

\usepackage[symbols,nogroupskip,sort=none]{glossaries-extra}

\usepackage{setspace}
\usepackage{mdframed}

\DeclareMathOperator{\rank}{rk}
\usepackage{tikz}
\usepackage{tikz-cd}
\usepackage{subfig}
\usetikzlibrary{arrows}
\usetikzlibrary{shapes.geometric}
\usetikzlibrary{decorations.pathreplacing,decorations.markings}

\usepackage{fancyhdr} 
\usepackage{sectsty}
\allsectionsfont{\sffamily\mdseries\upshape} 

\usepackage[nottoc,notlof,notlot]{tocbibind} 
\usepackage[titles,subfigure]{tocloft} 
\usepackage[super]{nth}
\usepackage[numbers]{natbib}
\usepackage{enumitem}

\newtheorem{Theo}{Theorem}[section]
\newtheorem{Def}{Definition}[section]

\newtheorem{Exp}{Example}[section]

\newtheorem{Lem}{Lemma}[section]
\newtheorem{Cor}{Corollary}[section]
\newtheorem{Pro}{Proposition}[section]

\newtheorem{Rk}{Remark}[section]
\newtheorem{Rks}{Remarks}[section]

\newcommand{\Z}{\mathbbm{Z}}

\title{An Improvement of the lower bound for the minimum number of link colorings by quandles}

\author{H. Abchir\\ {\scriptsize Hassan II University. 
		Casablanca, Morocco.}\\\scriptsize{e-mail: hamid.abchir@univh2c.ma} \\S. Lamsifer\\\scriptsize{Hassan II University. 
		Casablanca, Morocco.}\\\scriptsize{e-mail: soukaina.lamsifer-etu@etu.univh2c.ma}}

\begin{document}
\maketitle


\begin{abstract} We improve the lower bound for the minimum number of colors for linear Alexander quandle colorings of a knot given in Theorem 1.2 of ``Colorings beyond Fox: The other linear Alexander quandles'' (Linear Algebra and its Applications, Vol. 548, 2018). We express this lower bound in terms of the degree $k$ of the reduced Alexander polynomial of the considered knot. We show that it is exactly $k+1$ for L-space knots. Then we apply these results to torus knots and Pretzel knots $P(-2,3,2l+1)$, $l\ge 0$. We note that this lower bound can be attained  for some particular knots. Furthermore, we show that Theorem 1.2 quoted above can be extended to links with more that one component.
\end{abstract}


\section{Introduction}
The idea of coloring knots was initiated by R. Fox around 1960 (see \cite{fox1962quick}). He introduced colorings by dihedral quandles called $n$-colorings. Let $K$ be a $p$-colorable knot, where $p$ is prime and let $C_p(K)$ denote the minimal number of colors needed to color a diagram of $K$. Nakamura et al. proved in \cite{Naka_Nakani_Satoh} that $C_p(K)\ge\lfloor\log_2 p\rfloor +2$.\\
The problem of finding the minimum number of colors for $p$-colorable knots with primes up to 19 was investigated by many authors.  In 2009, S. Satoh showed in \cite{satoh20095} that $C_5(K)=4$. In 2010, K. Oshiro proved that $C_7(K)=4$ \cite{oshiro2010any}. In 2016, T. Nakamura, Y. Nakanishi and S. Satoh showed in \cite{satoh2016} that $C_{11}(K)=5$. In 2017, M. Elhamdadi and J. Kerr \cite{elhamdadi2016fox} and independently F. Bento and P. Lopes \cite{bento2017minimum} proved that $C_{13}(K)=5$. In 2020, H. Abchir, M. Elhamdadi and S. Lamsifer \cite{abch_elham_lamsifer} showed that $C_{17}(K)=6$. In 2022, Y. Han and B. Zhou showed that $C_{19}(K)=6$ \cite{Han_Zhou}.\\
The same problem may be studied for colorings by linear Alexander quandles which generalize dihedral ones. A linear Alexander quandle $\Lambda_{n,m}$ is a quandle whose underlying set is $\Z_n$, $n\ge 3$, endowed with the binary operation $x*y=mx+(1-m)y\,\,\mod n$, for some integer $m$ such that $(m,n)=1$. Let $L$ be a link which admits a non-trivial coloring by $\Lambda_{n,m}$, i.e. for which there exists a non-constant quandle homomorphism from the fundamental quandle $Q(L)$ to $\Lambda_{n,m}$. We denote by $mincol_{n,m}(L)$ the minimum number of colors needed to provide a non-trivial coloring of $L$. It is an interesting invariant of $L$. For a knot $K$, if $n$ is a prime integer $p$, L. Kauffman and P. Lopes showed in \cite{Kauff_Lopes} that
  $$2 + \lfloor \log_M p \rfloor \leq mincol_{p,m}(K),$$
  where $M= max \lbrace \vert m \vert, \vert m-1 \vert \rbrace$.\\
  We give an enhancement of the last result by proving the following theorem.
\begin{Theo}\label{MainTheo}
  Let $K$ be a knot. Let $\Delta_{K}^{0}(t)=\sum \limits_{\underset{}{i=0}}^{k} c_i t^i$ be the reduced Alexander polynomial of $K$. Let $m$ be an integer, such that $m>\displaystyle \max_{0\leq i \leq k} \lbrace |c_i|\rbrace +1$ and $p=\Delta_{K}^{0}(m)$ is an odd prime integer.
\begin{enumerate}
\item If $c_k=1$ and the penultimate non-zero coefficient is negative, then
  $$k+1 \leq mincol_{p,m}(K).$$
\item If $c_k>1$ or the penultimate non-zero coefficient is positive, then
  $$k+2\leq mincol_{p,m}(K).$$
\end{enumerate}
\end{Theo}

So, the lower bound stabilizes for suitable choices of $m$ and $p$ and no longer depends on these two integers. On the other hand, the lower bound we give comes from a topological invariant of the knot. In particular, it is exactly $k+1$ for an $L$-space knot, for any $m > 1$. This entails that if the torus knot $T(a,b)$ whose crossing number is $c(T(a,b))$ admits a non-trivial coloring by $\Lambda_{p,m}$, then $mincol_{p,m}(T(a,b))$ is bounded as follows: $$c(T(a,b))-(a-2)\le mincol_{p,m}(T(a,b))\le c(T(a,b)).$$
  
  Hence $mincol_{p,m}(T(2,b))=c(T(2,b))$. On the other hand, we show that for suitable choices of $m$ and $p$, $T(2,b)$ displays KH behavior, which means that $T(2,b)$ admits a reduced alternating diagram equipped with a non-trivial coloring by the quandle $\Lambda_{p,m}$ such that different arcs receive different colors. Furthermore, we show that if the Pretzel knot $P(-2,3,a)$ has a non-trivial coloring by a linear Alexander quandle $\Lambda_{p,m}$, then for suitable choices of $m$ and $p$ we have $a+4\le mincol_{p,m}(P(-2,3,a))$ and the equality holds for $m=2$. Finally, we show that Theorem 1.2 proved by Kauffman and Lopes in \cite{Kauff_Lopes} holds also for links with more than one component.\\

  The paper is organized as follows. In the second section, we recall the main tools we need to prove our results. In the third section we prove our main theorem. The fourth section is devoted to some applications. In the last section, we give a generalization of Theorem 1.2 in \cite{Kauff_Lopes}  to links.

\section{Preliminaries}
  In this section we give an overview of the main tools we need.\\

\noindent\textbf{Quandles.}

 \begin{Def}
A \textit{quandle}, is a non-empty set $Q$ equipped with a binary operation
$$
\begin{tabular}{cccc}
$* :$ & $Q\times Q$ & $\longrightarrow$ & $Q$  \\ 
 & $(x,y)$ & $\longmapsto$ & $x* y$ \\ 
\end{tabular} 
$$
satisfying the following three axioms:
\begin{enumerate}
\item For all $x\in Q$, $x* x =x$.
\item For all $y\in Q$, the map $R_y:Q\rightarrow Q$ defined by $R_y(x)=x* y$, $x\in Q$, is bijective.
\item For all $x, y, z \in Q$, $(x* y)* z=(x* z)* (y* z)$ (right self-distributivity).
\end{enumerate}
\end{Def}

\noindent We write $x*^{-1} y$ for $R_{y}^{-1}(x)$.\\
When needed, we will denote the quandle $Q$ by the pair $(Q,*)$.
\begin{Exp} \textit{Trivial quandle.}\\
Let $X$ be a non-empty set with the operation $x* y=x$ for any $x, y\in X$ (i.e. $R_{y}=id_{X}$, for any $y\in X$), is a quandle called the trivial quandle.
\end{Exp}
\begin{Exp}\textit{Dihedral quandle.}\\
Let $n$ be a positive integer. For $x$ and $y$ in $\Z_n$ (integers modulo $n$), define $x* y=2y-x \mod n$. The operation $*$ defines a quandle structure on $\Z_n$, called the dihedral quandle and is sometimes denoted $R_n$.
\end{Exp}

\begin{Exp}\textit{Alexander quandle.}\\
  An Alexander quandle $M$ is a $\Z[t,t^{-1}]$-module endowed with the following binary operation:
$$x* y=tx+(1-t)y\,\, {\rm for\,\, all}\,\, x,y\in M.$$
Note that we have $x*^{-1} y=t^{-1}x+(1-t^{-1})y$.\\
\end{Exp}

\begin{Exp}\label{Exp_4}\textit{Linear Alexander quandle.}\\
Let $n$ be an integer $n>1$, one can consider $\Lambda_n/(h)$ where $\Lambda_n=\Z_n[t,t^{-1}]$ and $h$ is a monic polynomial in $t$ (see \cite{Nelson}). If $h(t)=(t-m)$ and $m$ and $n$ are integers such that $\gcd{(m, n)}=1$, then we obtain what is called \textit{linear Alexander quandle} that we denote $\Lambda_{n,m}$. This amounts to considering the underlying set $\Z_n$ with the binary operation:
$$x* y=mx+(1-m)y \pmod n,\,\,\, {\rm for\,\, all}\,\, x,y\in \mathbb{Z}_n.$$
Note that we have $x*^{-1} y=m^{-1}x+(1-m^{-1})y \pmod n$.\\
If $m=-1$, we get the dihedral quandle $R_n$.
\end{Exp}

\begin{Def} Let $(Q_1,*_1)$ and $(Q_2,*_2)$ be two quandles. A map $f:Q_1\rightarrow Q_2$ is a \textit{quandle homomorphism} if it satisfies $$f(x*_1 y)=f(x)*_2 f(y),\ \ \forall x,y\in Q_1.$$
If $f$ is bijective, we say that $f$ is a quandle \textit{isomorphism}. If $f$ is bijective and $Q_1=Q_2$, the map $f$ is called a quandle \textit{automorphism}.
\end{Def}

\noindent\textbf{Coloring links by Alexander quandles}\\
Let $(Q,*)$ be a quandle and $D$ a diagram of an oriented link $L$. A coloring of $D$ by $Q$ is a map $\mathcal{C}$ from the set of arcs of $D$ denoted by $\mathcal{A}$ to $Q$, such that at each crossing of $D$, if the over-arc $\alpha_1$ is colored by $\mathcal{C}(\alpha_1)={y}$ and the incoming under-arc is colored by $\mathcal{C}(\alpha_2)=x$ then the outcoming under-arc is colored by $\mathcal{C}(\alpha_3)=x* y$ or $\mathcal{C}(\alpha_3)=x*^{-1} y$ according to the rule depicted in Fig.\ref{Fig.1}.
\begin{figure}[H]
\centering
	\includegraphics[scale=0.2]{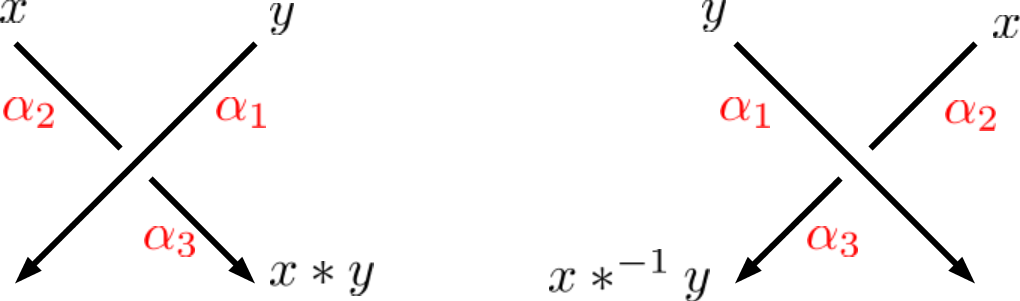} 
	\caption{Coloring conditions. \label{Fig.1}}
\end{figure}

If $Q$ is an Alexander quandle, by collecting the coloring conditions at all crossings of $D$, we get a homogeneous system of linear equations over $\Z[t,t^{-1}]$. The matrix associated to this system of equations is called the Alexander matrix. Its rows correspond to the crossings of $D$ and the columns correspond to the arcs of $D$. Each row has only three non-zero entries which are $t$, $1-t$ and $-1$. So on the one hand $(\lambda ,\dots ,\lambda)$ is a solution for any $\lambda\in\Z[t,t^{-1}]$ (trivial solutions), and on the other hand the determinant of the Alexander matrix is zero. Hence, it is easy to see that a non-trivial solution of the initial homogeneuous system corresponds to a non-trivial solution of the system of equations determined by the original matrix with one row and one column deleted. The determinant of this last submatrix is known to be the Alexander polynomial of the considered link denoted by $\Delta_L(t)$. Therefore, the existence of non-trivial solutions corresponds to working on the quotient of $\Z[t,t^{-1}]$ by $\Delta_L(t)$, which is a Laurent polynomial on the variable $t$ determined up to $\pm t^n$, for any integer $n$. We will use the reduced Alexander polynomial defined in \cite{Bae}.

\begin{Rks}
  \begin{enumerate}
  \item If $L$ is a knot $K$, then the reduced Alexander polynomial is exactly that given in the proof of Corollary 6.11 in \cite{Lickorish_book}, $\Delta_L(t)=c_0+c_1t+c_2t^2+\cdots +c_Nt^N$, where $N$ is even, $c_{N-r}=c_r$, $c_{\frac{N}{2}}$ is odd and $c_0>0$.
  \item If $L$ has $\mu$ components, $\mu\ge 2$, the reduced Alexader polynomial is obtained from the multivariable Alexander polynomial $\Delta_L(t_1,\dots ,t_\mu)$ by setting $t_i=t$ for each $i$. Recall that there is a relation between Alexander polynomial and multivariable Alexander polynomial as shown in Proposition 7.3.10 in \cite{Kawauchi_book}:$$\Delta_L(t)=(1-t)\Delta_L(t,\dots ,t).$$
  \end{enumerate}
\end{Rks}

\begin{Exp}
We consider the diagram $D$ of the knot $7_3$ whose arcs are labeled as shown in Fig.\ref{Fig.2}. By writing the coloring conditions illustrated in Fig.\ref{Fig.1} at each crossing $c_i$ of $D$, $1\leq i\leq 7$, we obtain the homogeneuous system of linear equations shown on the right side of Fig.\ref{Fig.2},
\begin{figure}[H]
\centering
	\includegraphics[scale=0.2]{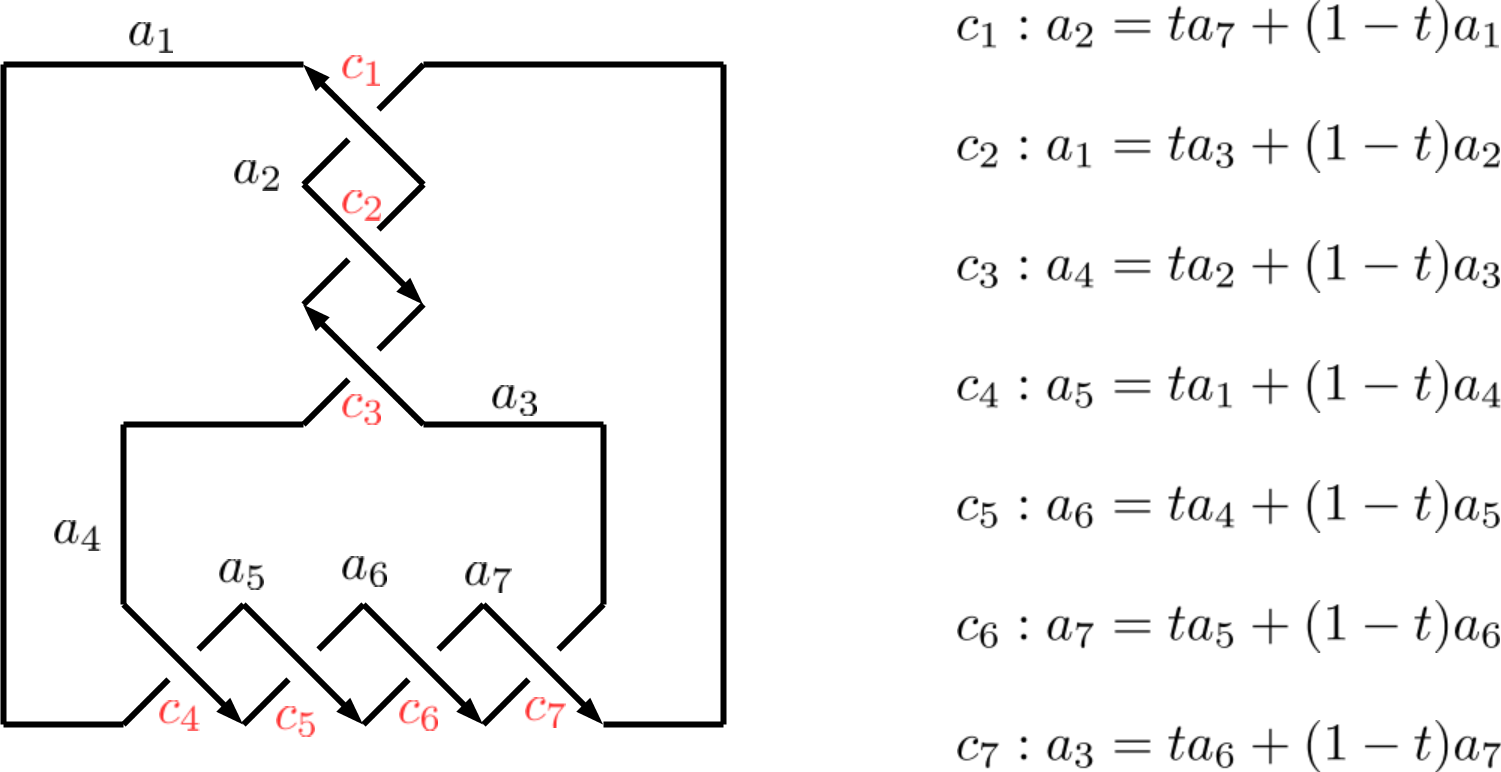} 
	\caption{The equations corresponding to the Alexander colorings of the knot $7_3$. \label{Fig.2}}
      \end{figure}
      
      \noindent we get the following Alexander matrix $A$:
      
\begin{equation*}
A = 
\begin{pmatrix}
a_1 & a_2 & a_3 & a_4 & a_5 & a_6 & a_7\\
\hline
1-t & -1 & 0 & 0 & 0 & 0 & t\\
-1 & 1-t & t & 0 & 0 & 0 & 0\\
0 & t & 1-t & -1 & 0 & 0 & 0\\
t & 0 & 0 & 1-t & -1 & 0 & 0\\
0 & 0 & 0 & t & 1-t & -1 & 0\\
0 & 0 & 0 & 0 & t & 1-t & -1\\
0 & 0 & -1 & 0 & 0 & t & 1-t\\
\end{pmatrix}
\end{equation*}

The determinant of the matrix $A$ is $0$. Let $M$ be the matrix obtained from $A$ by deleting the first row and the first column. The matrix $M$ is called the first minor matrix and its determinant is the Alexander polynomial of the knot $7_3$: $\det(M)=-2t+3t^2-3t^3+3t^4-2t^5.$

In order to obtain the reduced Alexander polynomial of the knot $7_3$, we multiply $\det(M)$ by $-t^{-1}$ and then we get
$$\Delta^0_{7_3}(t)=2-3t+3t^2-3t^3+2t^4.$$
\end{Exp}

\noindent\textbf{Coloring links by linear Alexander Quandles}\\
In practice, it is more interesting to color links by using finite quandles as linear Alexander quandles given in Example \ref{Exp_4}. These are the ones we will use to color links in this article. One can easily adapt what has been said in the general case of colorings by an Alexander quandle to this new setting.\\
So, if $m$ and $n$ are coprime integers, a coloring of a diagram $D$ of a link $L$ by the quandle $\Lambda_{n,m}$ is a map from the set $\mathcal{A}$ of arcs of $L$ to $\Lambda_{n,m}$ satisfying the coloring conditions in Fig. \ref{Fig.1}. It is easy to see that such non-trivial coloring exists if $n$ divides $\Delta_L^0(m)$. That coloring is called an $(n,m)$-coloring.\\
In this setting also, one can consider the minimum number of colors as it was done for Fox-colorings. We follow the definition given by Kauffman and Lopes in \cite{Kauff_Lopes}.
\begin{Def}\textbf{Minimum number of colors:}
Let $L$ be a link admitting non-trivial $(p,m)$-colorings. Let $D$ be a diagram of $L$ and let $n_{D,p,m}$ be the minimum number of colors it takes to equip $D$ with a non-trivial $(p,m)$-coloring. We let
$$mincol_{p,m}(L) = min\lbrace n_{D,p,m} |D \mbox{ is a diagram of } L\rbrace$$
and refer to it as the minimum number of colors for non-trivial $(p,m)$-colorings of $L$.
\end{Def}
The following Theorem gives an estimation of the minimum number of colors.
\begin{Theo}[\cite{Kauff_Lopes}]\label{Theo1}
Let $K$ be a knot i.e., a $1$-component link. Let $p$ be an odd prime. Let $m$ be an integer such that $K$ admits non-trivial $(p, m)$-colorings $\pmod p$. If $m \neq 2$ (or $m = 2$ but
$\Delta_{K}^{0} (m)\neq 0$) then
$$2 + \lfloor \log_M p \rfloor \leq mincol_{p,m}(K),$$
where $M= max \lbrace \vert m \vert, \vert m-1 \vert \rbrace$.
\end{Theo}
\section{An improvement of the lower bound for the minimum number of knot colorings by linear Alexander quandles}
In this section we prove that for suitable choice of the integer $m$, the lower bound of $mincol_{p,m}(K)$ stabilizes and no longer depends on $m$. This provides an interesting improvement of Theorem \ref{Theo1} when $p=\Delta_K^0(m)$.\\
We begin by proving the following lemma.

\begin{Lem}\label{Claim.1}
Let $K$ be a knot and $\Delta_{K}^{0}(t)=\sum \limits_{\underset{}{i=0}}^{k} c_i t^i$, its reduced Alexander polynomial. Let $m$ be an integer, $m>1$, and $p=\Delta_{K}^{0}(m)$. If $m>\displaystyle \max_{0\leq i \leq k} \lbrace |c_i|\rbrace+1$, then $2 + \lfloor \log_m p \rfloor$ is either $k+1$ or $k+2$.
\end{Lem}

\begin{proof}
  For convenience, we put $a_i=|c_i|$ and we write the reduced Alexander polynomial as follows: $$\Delta_{K}^{0}(t)=\sum \limits_{\underset{}{i=0}}^{k} \pm a_i t^i,$$ where $k$ is even, $a_{\frac{k}{2}}$ is odd, $a_0>0$ and for each $i=0,\dots ,k$, $a_i=a_{k-i}$. Hence $a_0=a_k\ne 0$ (see the proof of Corollary 6.11 in \cite{Lickorish_book}). Note that, $c_0=a_0=a_k=c_k$. The signs assigned to the non-null $a_i$ do not necessarily alternate as in the case of the knot $10_{145}$:
  $$\Delta_{10_{145}}^{0}(t)=1+t-3t^2+t^3+t^4.$$
  So we must distinguish two cases:
\begin{enumerate}
\item Suppose that the assigned signs to non-null $a_i$ alternate, two cases can occur:
  \begin{itemize}
  \item \underline{Case 1} If for each $i=0,\dots ,k$, $a_i\ne 0$, then we factor $\Delta_{K}^0(t)$ as follows
$$\Delta_{K}^0(t)=a_0+(t-a_1)t+(a_2-1)t^2+\cdots+(t-a_{k-1})t^{k-1}+(a_k-1)t^k .$$
For each $m> \displaystyle \max_{0\leq i \leq k} \lbrace a_i \rbrace$, by evaluating the reduced Alexander polynomial $\Delta_{K}^0(t)$ at $m$ we get:
$$p=\Delta_{K}^0(m)=a_0+(m-a_1)m+(a_2-1)m^2+\cdots+(m-a_{k-1})m^{k-1}+(a_k-1)m^k .$$
Since the integer $m>\displaystyle \max_{0\leq i \leq k} \lbrace a_i\rbrace$, then $p$ is a positive integer which can be written as follows 
$$
p= \left\{
    \begin{array}{ll}
        \sum \limits_{\underset{}{i=0}}^{k} d_i m^i & \mbox{if} \ a_k>1 \\
        \sum \limits_{\underset{}{i=0}}^{k-1} d_i m^i & \mbox{if} \ a_k=1,
    \end{array}
\right.
$$
where for all $0\leq i \leq k$
$$
d_i = \left\{
    \begin{array}{ll}
        m-a_i & \mbox{if i is odd } \\
        a_i-1 & \mbox{if i is even},\, i\ne 0\\
        a_0 & \mbox{if} \ i=0.
    \end{array}
\right.
$$
Since $\forall i \in \lbrace 0,..., k \rbrace$, $0\leq d_i < m$,
$d_k\geq 1$ if $a_k> 1$ and $d_{k-1}\geq1$ if $a_k=1$, then for all $m>\displaystyle \max_{0\leq i \leq k} \lbrace a_i \rbrace$
$$
p= \left\{
    \begin{array}{ll}
        (d_{k}d_{k-1}\cdot\cdot\cdot d_1 d_0)_m & \mbox{if} \ a_k>1 \\
        (d_{k-1}d_{k-2}\cdot\cdot\cdot d_1 d_0)_m & \mbox{if} \ a_k=1,
    \end{array}
\right.
$$
is the base $m$ expansion of the integer $p$. By Theorem 10.8.1 in \cite{adhikari} we have 
$$
\lfloor \log_mp \rfloor = \left\{
    \begin{array}{ll}
        k & \mbox{if} \ c_k =a_k> 1\\
        k-1 & \mbox{if} \ c_k=a_k=1.
    \end{array}
\right.
$$
\item \underline{Case 2} Now suppose that there exists at least one null coefficient $a_i$.\\
 We group each pair of two non-zero consecutive monomials starting with the first non-null monomial of degree greater than or equal to $1$ as follows
$$\Delta_{K}^0(t)=a_0-\underbrace{a_1t+a_2t^2}-\cdots-\underbrace{a_jt^j+a_lt^l}-\cdots-\underbrace{a_{k-1}t^{k-1}+a_kt^k}.$$
Suppose that there exists at least a binomial $-a_jt^j+a_lt^l$ such that $l> j+1$. Within these binomials we add expressions $t^r-t^r$ where $r$ corresponds to the missing degrees as follows
$$-a_jt^j+t^{j+1}-t^{j+1}+\cdots+t^{l-1}-t^{l-1}+a_lt^l.$$
We factor the expression above as follows
$$(t-a_j)t^j+(t-1)t^{j+1}+\cdots+(t-1)t^{l-2}+(t-1)t^{l-1}+(a_l-1)t^l,$$
and we factor the binomials of the form $-a_st^s+a_{s+1}t^{s+1}$ as follows
$$(t-a_s)t^s+(a_{s+1}-1)t^{s+1}.$$
For each $m>\displaystyle \max_{0\leq i \leq k} \lbrace a_i \rbrace$, by evaluating the reduced Alexander polynomial $\Delta_{K}^0(t)$ at $m$ we get:
\begin{align*}
p&=\Delta_{K}^0(m)\\ &=a_0+(m-a_1)m+(a_2-1)m^2+\cdots+(m-a_{j})m^{j}+(m-1)m^{j+1}\\
             &+\cdots+(m-1)m^{l-1}+(a_l-1)m^l+\cdots+(m-a_{k-1})m^{k-1}+(a_k-1)m^k.\\
\end{align*}
Since the integer $m>\displaystyle \max_{0\leq i \leq k} \lbrace a_i \rbrace$, then $p$ is a positive integer which can be written as follows 
$$
p= \left\{
    \begin{array}{ll}
        \sum \limits_{\underset{}{i=0}}^{k} d_i m^i & \mbox{if} \ a_k>1 \\
        \sum \limits_{\underset{}{i=0}}^{k-1} d_i m^i & \mbox{if} \ a_k=1,
    \end{array}
\right.
$$
where for all $i=0,\dots ,k$, $d_i\in\lbrace a_0,m-a_i, m-1, a_i-1\rbrace$.\\
Since $\forall i \in \lbrace 0,..., k \rbrace$, $0\leq d_i < m$, $d_k\geq 1$ if $a_k>1$ and $d_{k-1}\geq 1$ if $a_k=1$, then for all $m>\displaystyle \max_{0\leq i \leq k} \lbrace a_i \rbrace$
$$
p= \left\{
    \begin{array}{ll}
        (d_{k}d_{k-1}\cdot\cdot\cdot d_1 d_0)_m & \mbox{if} \ a_k>1 \\
        (d_{k-1}d_{k-2}\cdot\cdot\cdot d_1 d_0)_m & \mbox{if} \ a_k=1
    \end{array}
\right.
$$
is the base $m$ expansion of the integer $p$. Once again, by Theorem 10.8.1 in \cite{adhikari} we get
$$
\lfloor \log_mp \rfloor = \left\{
    \begin{array}{ll}
        k & \mbox{if} \ c_k=a_k > 1\\
        k-1 & \mbox{if} \ c_k=a_k=1.
    \end{array}
\right.
$$
\end{itemize}
\item Suppose that some consecutive non-null coefficients $a_i$ have opposite assigned signs. Since $\Delta_{K}^{0}(1)=1$, then there is at least one coefficient with negative assigned sign in $\Delta_{K}^0(t)$. By reading the monomials in ascending order of their degrees, each time we encounter an expression of the form
$$-a_jt^j-a_{j+1}t^{j+1}-\cdots-a_{j+l-1}t^{j+l-1}+a_{j+l}t^{j+l},$$
where $j+l\leq k$, $a_j\ne 0$ and $a_{j+l}\ne 0$, we factor it as follows without changing the other monomials:

\begin{equation}\label{Eq_1}
(t-a_j)t^j+(t-(a_{j+1}+1))t^{j+1}+\cdots+(t-(a_{j+l-1}+1))t^{j+l-1}+(a_{j+l}-1)t^{j+l}.
\end{equation}

Here we distinguish three cases depending on the last expression of the form (\ref{Eq_1}) occurring in the factored expression of $\Delta_K^0(t)$:
\begin{itemize}
\item If $j+l< k$, then for each $m>\displaystyle \max_{0\leq i \leq k} \lbrace a_i \rbrace$, we evaluate the factored reduced Alexander polynomial $\Delta_{K}^0(t)$ at $m$. Then we get a positive integer which can be written as follows 
$$p=\Delta_{K}^0(m)= \sum \limits_{\underset{}{i=0}}^{k} d_i m^i,$$
where for all $i=0,\dots , k$, $d_i\in\lbrace m-a_i, m-(a_i+1),a_i-1,a_i\rbrace$.\\
Since $\forall i $, $0\leq d_i < m$ and $d_k\geq 1$, then for all $m>\displaystyle \max_{0\leq i \leq k} \lbrace a_i \rbrace$, $$p= (d_{k}d_{k-1}\cdot\cdot\cdot d_1 d_0)_m$$ is the base $m$ expansion of the integer $p$. By Theorem 10.8.1 in \cite{adhikari} we have
$$\lfloor \log_mp \rfloor=k.$$
\item If $j+l=k$ and $j=k-1$, then for each $m>\displaystyle \max_{0\leq i \leq k} \lbrace a_i \rbrace$  we evaluate the factored reduced Alexander polynomial $\Delta_{K}^0(t)$ at $m$. We get  a positive integer which can be written as follows 
$$
p=\Delta_{K}^0(m) = \left\{
    \begin{array}{ll}
        \sum \limits_{\underset{}{i=0}}^{k} d_i m^i & \mbox{if} \ a_k>1  \\
        \sum \limits_{\underset{}{i=0}}^{k-1} d_i m^i & \mbox{if} \ a_k=1,
    \end{array}
\right.
$$
where for all $i=0,\dots , k$, $d_i\in\lbrace m-a_i, m-(a_i+1),a_i-1,a_i\rbrace$.\\
Since $\forall i \in \lbrace 0,..., k \rbrace$, $0\leq d_i < m$, $d_k\geq 1$ if $a_k>1$ and $d_{k-1}\geq1$ if $a_k=1$, then for all $m>\displaystyle \max_{0\leq i \leq k} \lbrace a_i \rbrace$,
$$
p= \left\{
    \begin{array}{ll}
        (d_{k}d_{k-1}\cdot\cdot\cdot d_1 d_0)_m & \mbox{if} \ a_k>1 \\
        (d_{k-1}d_{k-2}\cdot\cdot\cdot d_1 d_0)_m & \mbox{if} \ a_k=1,
    \end{array}
\right.
$$
is the base $m$ expansion of the integer $p$. Also, by Theorem 10.8.1 in \cite{adhikari} we have 
$$
\lfloor \log_mp \rfloor = \left\{
    \begin{array}{ll}
        k & \mbox{if} \ c_k =a_k> 1\\
        k-1 & \mbox{if} \ c_k=a_k=1.
    \end{array}
\right.
$$
\item If $j+l=k$ and $j<k-1$ then for each $m>\displaystyle \max_{0\leq i \leq k} \lbrace a_i\rbrace +1$, we evaluate the factored reduced Alexander polynomial $\Delta_{K}^0(t)$ at $m$. Then $\Delta_{K}^0(m)$ is a positive integer which can be written as follows 
$$
\Delta_{K}^0(m) = \left\{
    \begin{array}{ll}
        \sum \limits_{\underset{}{i=0}}^{k} d_i m^i & \mbox{if} \ a_k>1  \\
        \sum \limits_{\underset{}{i=0}}^{k-1} d_i m^i & \mbox{if} \ a_k=1,
    \end{array}
\right.
$$
where for all $i=0,\dots ,k$, $d_i\in\lbrace m-a_i, m-(a_i+1),a_i-1,a_i\rbrace$.\\
Since $\forall i \in \lbrace 0,..., k \rbrace$, $0\leq d_i < m$, $d_k\geq 1$ if $a_k>1$ and $d_{k-1}\geq1$ if $a_k=1$, then for all $m>\displaystyle \max_{0\leq i \leq k} \lbrace a_i \rbrace+1$
$$
p=\Delta_{K}^0(m) = \left\{
    \begin{array}{ll}
        (d_{k}d_{k-1}\cdot\cdot\cdot d_1 d_0)_m & \mbox{if} \ a_k>1 \\
        (d_{k-1}d_{k-2}\cdot\cdot\cdot d_1 d_0)_m & \mbox{if} \ a_k=1,
    \end{array}
\right.
$$
is the base $m$ expansion of the integer $p$. Then by Theorem 10.8.1 in \cite{adhikari}, we have
$$
\lfloor \log_mp \rfloor = \left\{
    \begin{array}{ll}
        k & \mbox{if} \ c_k =a_k> 1,\\
        k-1 & \mbox{if} \ c_k=a_k=1.
    \end{array}
\right.
$$
\end{itemize}
\end{enumerate}
\end{proof}

\begin{Rk}\label{Rk_MainTheo}
  The proof shows that the condition $m>\displaystyle \max_{0\leq i \leq k} \lbrace |c_i|\rbrace+1$ is needed in the only one case where the non-null coefficients do not alternate and the two penultimate non-null coefficients have negative signs (as in the last subcase studied in the proof of the last Lemma). Otherwise the weaker condition $m>\displaystyle \max_{0\leq i \leq k} \lbrace |c_i|\rbrace$ suffices.
\end{Rk}

The lemma \ref{Claim.1} leads to the following improvement of Theorem \ref{Theo1}.
\begin{Theo}\label{MainTheo}
  Let $K$ be a knot. Let $\Delta_{K}^{0}(t)=\sum \limits_{\underset{}{i=0}}^{k} c_i t^i$ be the reduced Alexander polynomial of $K$. Let $m$ be an integer, such that $m>\displaystyle \max_{0\leq i \leq k} \lbrace |c_i|\rbrace +1$ and $p=\Delta_{K}^{0}(m)$ is an odd prime integer.
\begin{enumerate}
\item If $c_k=1$ and the penultimate non-zero coefficient is negative, then
  $$k+1 \leq mincol_{p,m}(K).$$
\item If $c_k>1$ or the penultimate non-zero coefficient is positive, then
  $$k+2\leq mincol_{p,m}(K).$$
\end{enumerate}
\end{Theo}

\begin{proof}
  We apply Theorem \ref{Theo1} and then we apply Lemma \ref{Claim.1} to the obtained lower bound $(2 + \lfloor \log_m p \rfloor)$.
\end{proof}
\begin{Exp}
The trefoil knot has the reduced Alexander polynomial $\Delta_{3_1}^0(t)=1-t+t^2$. For any integer $m$, if $p=\Delta_{3_1}^0(m)$ is an odd prime, then the trefoil knot admits non-trivial $(p,m)$-colorings. Furthermore, since the leading coefficient of $\Delta_{3_1}^0(t)$ is $1$ and its penultimate coefficient is negative, then by Theorem \ref{MainTheo} and Remark \ref{Rk_MainTheo}, for any $m>1$, we have
$mincol_{p,m}(3_1)\ge 3$.\\
Since the crossing number of the knot $3_1$ is $3$ then for all $m>1$,  $mincol_{p,m}(3_1)= 3$.\\
The following table displays some prime values for the reduced Alexander polynomial $\Delta_{3_1}^0(m)$.

\begin{table}[htbp]
  \centering
  \begin{tabular}[c]{|p{2cm}|p{1cm}|p{1cm}|p{1cm}|p{1cm}|p{1cm}|p{1cm}|p{1cm}|p{1cm}|}
 \hline
 $m$ & $2$ & $3$ & $4$ & $6$ & $7$ & $9$ & $13$ & $15$\\
 \hline
 $\Delta_{3_1}^0(m)$ & $3$ & $7$ & $13$& $31$ & $43$  & $73$ & $157$ & $211$\\
\hline
  \end{tabular}
  \caption{Prime values of $\Delta_{3_1}^0(m)$.}
\end{table}
\end{Exp}

\section{Some applications}
For any rational homology 3-sphere $M$, the rank of  the Heegaard Floer homology $\widehat{HF}(M)$ is lower bounded by the order of $H_1(M;\Z)$ \cite{ozsv_szab}. If the rank of $\widehat{HF}(M)$ is equal to the order of $H_1(M;\Z)$ then $M$ is called an \textit{L-space} \cite{ozsv_szab}.\\
A knot $K$ in $S^3$ is called an \textit{L-space knot} if some positive surgery on $K$ gives a 3-manifold that is an L-space. Torus knots are L-space knots (see \cite{Osvath_2} page 379).\\

Our main Theorem \ref{MainTheo} provides the interesting following corollary.
\begin{Cor}\label{Cor_MainTheo}
Let $K$ be an L-space knot. If $m$ is an integer such that $m>1$ and $p=\Delta_{K}^{0}(m)$ is an odd prime, then
$$k+1 \leq mincol_{p,m}(K).$$
\end{Cor}

\begin{proof}
We know that each non-zero coefficient of the Alexander polynomial of any L-space knot  is $\pm 1$ and they alternate in sign \cite{ozsv_szab}. It is easy to see that this holds also for the reduced Alexander polynomial. Then by Theorem \ref{MainTheo} and Remark \ref{Rk_MainTheo}, we have for each integer $m>1$ such that $p=\Delta_{K}^{0}(m)$ is an odd prime
$$k+1 \leq mincol_{p,m}(K).$$
\end{proof}

By applying Corollary \ref{Cor_MainTheo} to some families of knots, we get the following propositions.\\

\noindent{\textbf{Torus knots.}}\\
We recall that if $T_{a,b}$ is a torus knot, then the non-zero coefficients of the reduced Alexander polynomial $\Delta_{T_{a,b}}^{0}(t)$ are all $\pm 1$, and they alternate in sign \cite{Osvath_2}.
\begin{Pro}\label{Pro-1_MainTheo}
Let $T_{a,b}$ be a torus knot. Let $m$ be an integer such that, $m>1$ and $p=\Delta_{T_{a,b}}^{0}(m)$ is an odd prime, then 

$$c(T_{a,b})-(a-2)\leq mincol_{p,m}(T_{a,b})\leq c(T_{a,b}),$$
where $c(T_{a,b})$ is the crossing number of the torus knot $T_{a,b}$.\\
In particular

$$mincol_{p,m}(T_{2,b})=c(T_{2,b}).$$
\end{Pro}

\begin{proof} We note that $\Delta_{T_{-a,b}}^{0}(t)=\Delta_{T_{a,-b}}^{0}(t)=\Delta_{T_{-a,-b}}^{0}(t)=\Delta_{T_{a,b}}^{0}(t)$. Then, by Corollary \ref{Cor_MainTheo}, for any $m$ as in the statement of the proposition we have 
\begin{equation}\label{Eq_2}
k+1\leq mincol_{p,m}(T_{a,b}),
\end{equation}
where $k$ is the degree of $\Delta_{T_{a,b}}^{0}(t)$. Since the torus knot $T_{a,b}$ is equivalent to the torus knot $T_{b,a}$, we can restrict our computation to the case where $a<b$.\\
On the other hand, the reduced Alexander polynomial of the torus knot can be written
$\Delta_{T_{a,b}}(t)=\dfrac{f(t^b)}{f(t)}$ where $f(t)= 1+\cdot\cdot\cdot+t^{a-1}$ \cite{Agle}. It follows that
$$
\begin{array}[c]{ccc}
k &=&(a-1)(b-1)\\
  &=&ab-a-b+1\\
  &=&b(a-1)-(a-1).\\
\end{array}
$$
It has been proven by Murasugi in \cite{Murasugi} that $c(T_{a,b})=min\lbrace a(b-1),b(a-1) \rbrace$ which is exactly $b(a-1)$ since we assumed that $a<b$. Then $k=c(T_{a,b})-(a-1)$ . By replacing $k$ in (\ref{Eq_2}), we get
$$c(T_{a,b})-(a-2)\leq mincol_{p,m}(T_{a,b}).$$
The right inequality is obvious. Finally we have 

$$c(T_{a,b})-(a-2)\leq mincol_{p,m}(T_{a,b})\leq c(T_{a,b}).$$
If $a=2$ then $mincol_{p,m}(T_{2,b})= c(T_{2,b})$.
\end{proof}
\noindent\textbf{Pretzel knots $P(-2, 3, 2l + 1)$.}\\
Lidman and Moore showed that $P(-2, 3, 2l + 1)$, $l\ge 0$, are the only L-space Pretzel knots \cite{Lidman_Moore}.\\
In what follows, we prove that if we color a diagram of the Pretzel knot $P(-2, 3, 2l + 1)$, $l \geq 0$, by the linear Alexander quandle $\Z_p[t]/(t-2)$ where $p$ is an odd prime, then the lower bound in Theorem \ref{MainTheo} is reached. We need the following lemma.

\begin{Lem}\label{Claim.4}
For any odd integer $a\geq 1$, the Alexander polynomial of the Pretzel knot $K=P(-2,3,a)$ can be written as follows

$$\Delta_{K}(t)=1-t+\sum \limits_{\underset{}{i=3}}^{a} (-1)^{i+1} t^i-t^{a+2}+t^{a+3}.$$
\end{Lem}

\begin{proof}
Since $a$ is an odd integer then it can be expressed as $a=2l+1$, where $l\ge 0$. We will proof by induction that, for all  $l\in \Z$,

$$\Delta_{P(-2,3,2l+1)}(t)=1-t+\sum \limits_{\underset{}{i=3}}^{2l+1} (-1)^{i+1} t^i-t^{2l+3}+t^{2l+4}.$$
It is known that the Alexander polynomial of the Pretzel knot $P(p,q,-2)$, where $p$ and $q$ are odd integers, is given by (see \cite{Hironaka})
\begin{equation}\label{Alexander_poly_Pretzel_knots}
\Delta_{P(p,q,-2)}(t) =\dfrac{1+2t+t^{1+p}+t^{1+q}- t^3 -t^{p+q} + t^{p+2} + t^{q+2} + 2t^{p+q+2} + t^{3+p+q}}{(1+t)^3}.
\end{equation}
\\
Note that $P(p,q,-2)$ and $P(-2,p,q)$ are equivalent (see \cite{Kawauchi_book}, Paragrapah 2.3).\\
If $l=1$ 

$$\Delta_{P(-2,3,3)}(t)=\dfrac{1+2t-t^3+2t^4+2t^5-t^6+2t^8+t^9}{(1+t)^3}=1-t+t^3-t^5+t^6.$$
Now suppose that the formula is right for $l$, i.e.
$$\Delta_{P(-2,3,2l+1)}(t)=1-t+\sum \limits_{\underset{}{i=3}}^{2l+1} (-1)^{i+1} t^i-t^{2l+3}+t^{2l+4},$$
and show that

$$\Delta_{P(-2,3,2l+3)}(t)=1-t+\sum \limits_{\underset{}{i=3}}^{2l+3} (-1)^{i+1} t^i-t^{2l+5}+t^{2l+6}.$$
By using formula (\ref{Alexander_poly_Pretzel_knots}), we can write
\begin{align*}
\Delta_{P(-2,3,2l+3)}(t) &=\dfrac{1+2t+t^{2l+4}+t^4-t^3-t^{2l+6}+t^{2l+5}+t^5+2t^{2l+8}+t^{2l+9}}{(1+t)^3}\\
             &=\dfrac{1+2t+t^4-t^3+t^5+t^2(t^{2l+2}-t^{2l+4}+t^{2l+3}+2t^{2l+6}+t^{2l+7})}{(1+t)^3}\\
             &=\dfrac{1+2t-t^3+t^4+t^5-t^2(1+2t-t^3+t^4+t^5)}{(1+t)^3}+t^2\Delta_{P(-2,3,2l+1)}(t)\\
             &=\dfrac{1+2t-t^2-3t^3+t^4+2t^5-t^6-t^7}{(1+t)^3}+t^2\Delta_{P(-2,3,2l+1)})(t)\\
             &=1-t-t^2+2t^3-t^4+t^2(1-t+\sum \limits_{\underset{}{i=3}}^{2l+1} (-1)^{i+1} t^i-t^{2l+3}+t^{2l+4})\\
             &=1-t+t^3-t^4+\sum \limits_{\underset{}{i=3}}^{2l+1} (-1)^{i+1} t^{i+2}-t^{2l+5}+t^{2l+6}\\
             &=1-t+t^3-t^4+\sum \limits_{\underset{}{i=5}}^{2l+3} (-1)^{i-1} t^{i}-t^{2l+5}+t^{2l+6}\\
             &=1-t+\sum \limits_{\underset{}{i=3}}^{2l+3} (-1)^{i-1} t^{i}-t^{2l+5}+t^{2l+6}.\\
\end{align*}
Since $\forall i$, $(-1)^{i+1}=(-1)^{i-1}$ then
$$\Delta_{P(-2,3,2l+3)}(t)=1-t+\sum \limits_{\underset{}{i=3}}^{2l+3} (-1)^{i+1} t^i-t^{2l+5}+t^{2l+6}.$$
\end{proof}

\begin{Pro}
Let $K=P(-2,3,a)$ be a pretzel knot where $a$ is an odd positive integer. Let $m> 1$ be an integer such that $p=\Delta_{K}^{0}(m)$ is an odd prime. If $\Delta_{K}^{0}(m)\neq 0$ then
$$a+4\leq mincol_{p,m}(K).$$
In particular, if $m=2$ then
$$a+4 = mincol_{p,m}(K).$$
\end{Pro}

\begin{proof}
By Theorem \ref{Cor_MainTheo}, we have 
\begin{equation}
k+1\leq mincol_{p,m}(K),
\end{equation}
where $k$ is the degree of the reduced Alexander polynomial.
By Lemma \ref{Claim.4} the degree of the reduced Alexander polynomial of the Pretzel knot $K=P(-2,3,a)$ is $k=a+3$. Finally, we have
$$a+4\leq mincol_{p,m}(K).$$
If $m=2$, it is easy to see that $\Delta_{P(-2,3,a)}^{0}(2)\ne 0$. Furthermore, if $p=\Delta_{P(-2,3,a)}^{0}(2)$ is prime, consider the diagram of the Pretzel knot $P(-2,3,a)$ shown in Fig. \ref{Fig.6}, we prove that this diagram has a non-trivial $(p,m)$-coloring using exactly $a+4$ colors. We assign the colors $x,y,z,w \in \Z_p$ to the four arcs in the left tower.   
\begin{figure}[H]
\centering
	\includegraphics[scale=0.15]{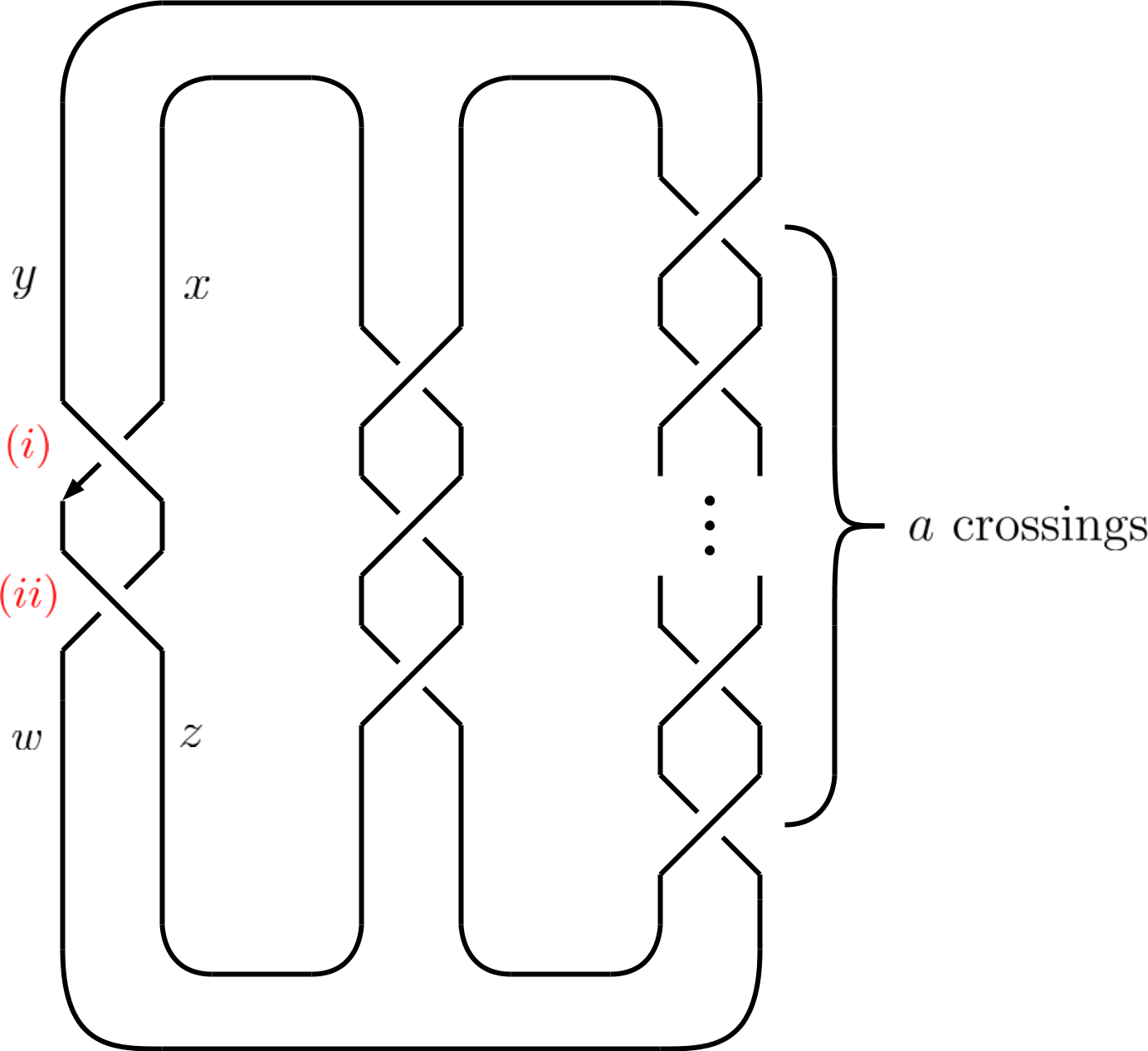} 
	\caption{Pretzel knot $P(-2,3,a)$. \label{Fig.6}}
\end{figure}
Using the coloring conditions associated with the crossings $(i),(ii)$ we get the following system of equations.
$$
\left\{
    \begin{array}{ll}
        z=mx+(1-m)y \pmod p\\
        y=mw+(1-m)z \pmod p.\\
    \end{array}
\right.
$$
Since we assumed that $P(-2,3,a)$ has a non-trivial $(p,m)$-coloring, as for Fox colorings (\cite{Ge} Lemma 2.10), there exists an automorphism of the considered linear Alexander quandle $\Z_p$ providing a new $(p,m)$-coloring containing colors $0$ and $1$. So we can assume that $x=1$ and $y=0$.
$$
\left\{
    \begin{array}{ll}
        z=m \pmod p\\
        0=mw+(1-m)z \pmod p.\\
    \end{array}
\right.
$$
If we replace $z$ in the second equation we obtain $mw=-m(1-m) \pmod p\Rightarrow w=(m-1)\pmod p$. If $m=2$ then $w=1=x$. The number of arcs of the Pretzel knot $P(-2,3,a)$ is equal to $a+5$ and, since two arcs are equipped with the same color, we can conclude that $mincol_{p,2}(K)=a+4$.
\end{proof}

\noindent\textbf{About the Kauffman-Harary conjecture}

We deduce another result from Proposition \ref{Pro-1_MainTheo} regarding what was called Kauffman-Harary conjecture proved by Mattman and Solis in \cite{Mattman_Solis} for Fox colorings.
\begin{Theo}[\cite{Mattman_Solis}]
  Let $D$ be a reduced, alternating diagram of the knot $K$ having prime determinant $p$. Then every non-trivial $p$-coloring of $D$ assigns different colors to different arcs.
\end{Theo}

Kauffman and Lopes extended the KH-behavior property to colorings using linear Alexander quandles as follows \cite{Kauff_Lopes}.

\begin{Def}
We say an alternating knot $K$ displays KH behavior mod $(p,m)$, if there is an integer $m$ such that $1 < m < \Delta_{K}^{0}(m)=p$, where $p$ is prime, and for some reduced alternating diagram of $K$ equipped with a non-trivial $(p, m)$-coloring, different arcs receive different colors. Otherwise we say that $K$ displays anti-KH behavior.
\end{Def}

They noted that there exist some knots displaying anti-KH behavior. That means that for such knots, any alternating reduced diagram equipped with non-trivial $(p,m)$-coloring associated to prime determinant have at least two distinct arcs which receive the same color. So, it is interesting to investigate if there are some families of knots that display KH behavior mod $(p,m)$.\\

Proposition \ref{Pro-1_MainTheo} allows to answer for torus knots $ T_{2,b}$.

\begin{Pro}
The torus knot $T_{2,b}$ displays KH-behavior mod $(p,m)$.
\end{Pro}

\begin{proof}
Suppose $T_{2,b}$ displays anti-KH behavior i.e. there is an integer $m$ such that $1 < m < \Delta_{T_{2,b}}^{0}(m)=p$ where $p$ is prime and for any reduced alternating diagram of $T_{2,b}$ equipped with non-trivial $(p,m)$-coloring there exists at least two distinct arcs with the same color. This implies that the number of colors is strictly less than the crossing number of $T_{2,b}$ (contradiction because we proved that $mincol_{p,m}(T_{2,b})=c(T_{2,b})$).
\end{proof}

\section{Generalization of Theorem 2.1 to links}

It is worth mentioning that the proof of Theorem \ref{Theo1} given in \cite{Kauff_Lopes} cannot be naturally extended to links with non-zero determinant. However, it has already been proved that there is an analogous theorem for links in the case of Fox-colorings \cite{Ichi_Matsu}. In view of this, it is natural to ask what can we say for links colored by linear Alexander quandles. This section is devoted to answer this question. We show the following theorem.

\begin{Theo}\label{Theo.4}
Let $L$ be a link whose reduced Alexander polynomial $\Delta^0_{L}(t) \neq 0$. Let $m$ be an integer and $p$ a prime factor of  $\Delta_{L}^{0}(m)$ such that $L$ admits non-trivial $(p,m)$-colorings. Suppose that $\Delta_{L}^{0}(m) \neq 0$ then
$$2 + \lfloor \log_M p \rfloor \leq mincol_{p,m}(L),$$
where $M= max \lbrace \vert m \vert, \vert m-1 \vert \rbrace$.
\end{Theo}
The  proof we give is inspired by that given in \cite{Ichi_Matsu} for effective Fox $n$-colorings.\\

Let $D$ be a diagram of $L$ equipped with a non-trivial $(p,m)$-coloring. Let $q$ be the number of arcs of $D$, $\{a_1,\ldots, a_q\}$ be the set of arcs of $D$ and $\{c_1,\ldots, c_q\}$ be the set of crossings of $D$. Let $A$ be the corresponding coloring matrix. For each $i$, $1 \leq i \leq q$, let $x_i$  be the color assigned to the arc $a_i$. Note that
\begin{itemize}
\item Since the considered $(p,m)$-coloring is non-trivial then the vector $X_0={}^t(x_1,\ldots,x_q)\in \Z^q$ is a non-trivial solution of the congruence equation $AX \equiv 0 \pmod p$.
\item Each row of $A$  contains only three non-zero entries which are $-m$, $m-1$ and $1$.
\item Since any first minor of $A$ is the reduced Alexander polynomial of $L$ (up to multiplication by $\pm t^{n}$) valued at $m$ is assumed not equal to $0$, then $\rank A= q-1$.
\end{itemize}
Let $d$ be the number of pairwise distinct colors $x_i$, $1 \leq i \leq d$. We agree to associate the same label to the arcs having the same color. So, we get the label set $\{a_1,...,a_d\}$. By writing the coloring condition using the new labelings at each crossing of $D$, we get a new system of homogeneous linear equations $\pmod p $ whose matrix is denoted by $A_1$. This system has $q$ equations in $d$ unknowns.\\

Note that the matrix $A_1$ is exactly the matrix obtained in the procedure described in the proof of Theorem 1.1 in \cite{Ichi_Matsu} which consists of adding the \(j\)\textsuperscript{th} column to the \(i\)\textsuperscript{th} column and then deleting the \(j\)\textsuperscript{th} one, for all $i$ and $j$, $1\leq i<j \leq q$, such that $x_i=x_j$.\\

The matrix $A_1$ satisfies the following properties.
\begin{itemize}
\item The vector $Y_0={}^t(y_1,...,y_d)$ obtained from $X_0$ by keeping only one representative for each color occurring in $X$ provides a non-trivial solution to $A_1Y \equiv 0 \pmod p $.
\item The only non-zero entries in any row of $A_1$ are still the same as those of any row of $A$, namely $-m$, $m-1$ and $1$. This is true because the sum of two of these entries cannot be associated to a same color occurring at a same crossing, otherwise the three colors associated to arcs of that crossing will be the same. 
\item $\rank A_1 =d-1$: We note that $\rank A_1 \le d-1$. Suppose that $\rank A_1 <d-1$, this implies that $\rank A <q-1$. This contradicts the fact that $\rank A=q-1$.
\end{itemize}
Now, we consider $d-1$ linearly independent row vectors of $A_1$. Then we get a $(d-1)\times d$ matrix $A_2$. Note that $Y_0={}^t(y_1,...,y_d)$ gives a solution of $A_2Y \equiv 0 \pmod p$ because the rows of $A_2$ form a subset of those of $A_1$. 
Since the non-zero entries on each row of $A_2$ are $-m$, $m-1$ and $1$ then by adding all the columns to the last column we obtain a $(d-1)\times d$ matrix $A_3$ whose last column contains only zeros. By applying Lemma 2.1 in \cite{Naka_Nakani_Satoh}, we see that $Z_0={}^t(y_1-y_d,...,y_{d-1}-y_d,0)$ is a non-trivial solution of $A_3Z \equiv 0 \pmod p$.
\begin{Lem}\label{Lem_Theo}
Let $p$ be an odd prime and $m$ be an integer. Let $L$ be a link
admitting non-trivial $(p, m)$-colorings. Let $B$ be the square matrix $(d-1)\times(d-1)$ obtained from $A_3$ by deleting the last column, then we have the following:
\begin{enumerate}[label=(\roman*)]
\item $\vert \det(B)\vert \leq  M^{d-1}$, where $M= max\lbrace\vert m \vert, \vert m-1 \vert\rbrace$.
\item $det(B)$ is divisible by $p$.
\end{enumerate} 
\end{Lem}
\begin{proof}
\begin{enumerate}[label=(\roman*)]
\item Since each row of $B$ contains at least two non-zero entries belonging to the set $\{-m,m-1,1\}$, by applying Lemma 3.1 in \cite{Kauff_Lopes} we have $\vert \det(B)\vert \leq M^{d-1}$, where $M= max\lbrace\vert m \vert, \vert m-1 \vert\rbrace$.
\item We denote $V_0={}^t(y_1-y_d,...,y_{d-1}-y_d)$ which is a non-trivial solution of $BV \equiv 0 \pmod p$ then $\det B=0 \pmod p$.
\end{enumerate}
\end{proof}
\begin{proof}[Proof of Theorem \ref{Theo.4}]
By lemma \ref{Lem_Theo} we have,
$$ p \leq \vert det B \vert \leq M^{d-1},$$
where $M= max\lbrace\vert m \vert, \vert m-1 \vert\rbrace$. We remove $\vert det(B)\vert$ from the inequalities and then we get $p<M^{d-1}$. By applying logarithm base $M$ we obtain
$$ d > 1+\log_Mp.$$
\end{proof}

\begin{Exp}
We consider the link $L4a1\lbrace1\rbrace$ depicted in the following figure \ref{Fig.8}.
\begin{figure}[H]
\centering
	\includegraphics[scale=0.13]{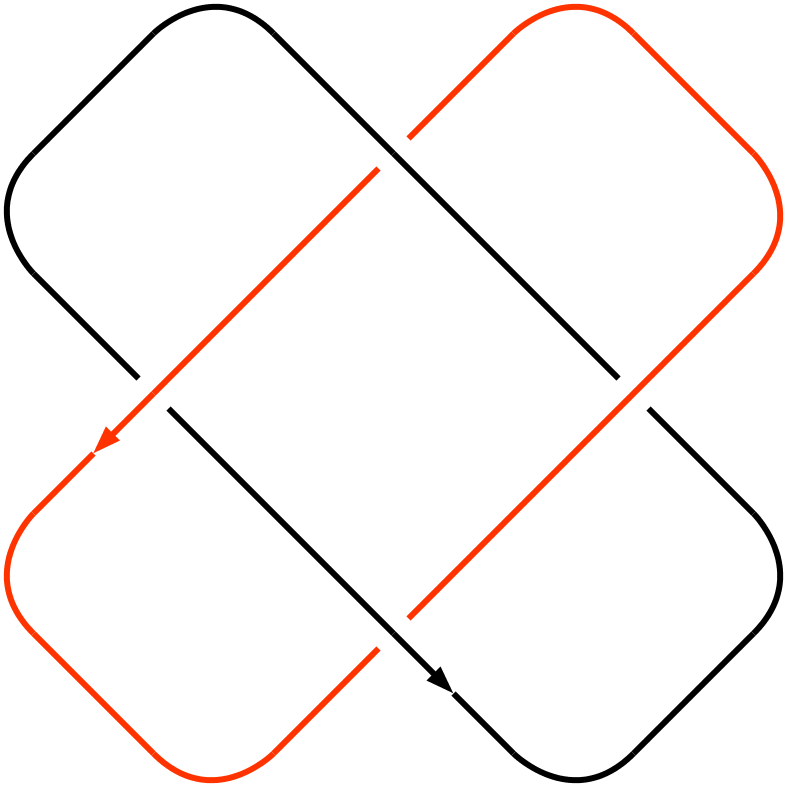} 
	\caption{The link $L4a1\lbrace1\rbrace$. \label{Fig.8}}
\end{figure}
The reduced Alexander polynomial of the link $L4a1\lbrace1\rbrace$ is $\Delta_{L4a1\lbrace1\rbrace}^0(t)=1+t^2$. For any integer $m$, if $p=\Delta_{L4a1\lbrace1\rbrace}^0(m)$ is a prime integer, then the link $L4a1\lbrace1\rbrace$ admits non-trivial $(p,m)$-colorings. For any $m>1$ we have $\Delta_{L4a1\lbrace1\rbrace}^0(m)=1+m^2$. This is a positive integer which can be written as follows
$$\Delta_{L4a1\lbrace1\rbrace}^0(m) = \sum \limits_{\underset{}{i=0}}^{2} d_i m^i,$$
where $d_0=1$, $d_1=0$ and $d_2=1$. Since $d_2\geq 1$ and for any $0\leq i \leq 2$, $0\leq d_i <m$, then by Theorem 8.10.1 in \cite{adhikari} we have $\lfloor \log_m{\Delta_{L4a1\lbrace1\rbrace}^{0}(m)}\rfloor =2$. Hence, for all integers $m>1$,
$$mincol_{p,m}(L4a1\lbrace1\rbrace)\geq 4.$$
Since the crossing number of the link $L4a1\lbrace1\rbrace$ is $4$, then for any integer $m>1$,
$$mincol_{p,m}(L4a1\lbrace1\rbrace)= 4.$$
The following table displays some prime values for the reduced Alexander polynomial $\Delta_{L4a1\lbrace1\rbrace}^0(m)$.
\begin{table}[htbp]
  \centering
  \begin{tabular}[c]{|p{2cm}|p{1cm}|p{1cm}|p{1cm}|p{1cm}|p{1cm}|p{1cm}|p{1cm}|p{1cm}|}
  \hline
 $m$ & $2$ & $4$ & $6$ & $10$ & $14$ & $16$ & $20$ & $24$\\
 \hline
 $\Delta_{L4a1\lbrace1\rbrace}^0(m)$ & $5$ & $17$ & $37$& $101$ & $197$  & $257$ & $401$ & $577$ \\
\hline
  \end{tabular}
  \caption{Prime values of $\Delta_{L4a1\lbrace1\rbrace}^0(m)$.}
\end{table}
\end{Exp}

\end{document}